\documentclass[12pt,a4paper,reqno]{amsart}
\usepackage[margin=2cm]{geometry}

\usepackage{xcolor}
\usepackage{amssymb}
\usepackage{amsmath,amsthm}
\usepackage{mathrsfs}

\usepackage{latexsym,euscript,dsfont}
\usepackage{bbm}
\usepackage{colortbl}
\usepackage{cite}
\usepackage{amsmath,amsthm,amsfonts,amssymb}
\usepackage{latexsym}
\usepackage{enumerate}
\usepackage{mathrsfs}
\def\bc{\begin{center}}
\def\ec{\end{center}}
\def\be{\begin{equation}}
\def\ee{\end{equation}}
\def\F{\mathcal F}
\def\N{\mathbb N}

\def\H{\mathcal H}

\def\R{\mathbb R}

\newcommand\hdim{\dim_{\mathrm H}}

\newtheorem{lem}{Lemma}[section]

\newtheorem{dfn}[lem]{Definition}
\newtheorem{pro}[lem]{Proposition}
\newtheorem{thm}[lem]{Theorem}

\newtheorem{rem}{Remark}
\numberwithin{equation}{section}

\usepackage{colortbl}

\begin{document}
\title[ Shrinking target problem for beta dynamical systems ]{Higher dimensional shrinking target problem for beta dynamical systems}


 \author[M. Hussain]{Mumtaz Hussain}
\address{Mumtaz Hussain,  La Trobe University, POBox199, Bendigo 3552, Australia. }
\email{m.hussain@latrobe.edu.au}


\author[W. Wang]{Weiliang Wang}
\address{Weiliang Wang, Department of Mathematics, West Anhui University, Liu'an, Anhui 237012, China}
\email{weiliang\_wang@hust.edu.cn}

\keywords{Beta-expansions, shrinking target problem, Hausdorff dimension}
\subjclass[2010]{Primary 11K55; Secondary 28A80, 11J83, 11K60, 37C45, 37A45}
\maketitle
\begin{abstract}
We consider the two dimensional shrinking target problem in the beta dynamical system  for general $\beta>1$ and with the general error of approximations.  Let $f, g$ be two positive continuous functions.
For any $x_0,y_0\in[0,1]$, define the shrinking target set
 $$
E(T_\beta, f,g):=\left\{(x,y)\in [0,1]^2:  \begin{array}{ll} |T_{\beta}^{n}x-x_{0}|<e^{-S_nf(x)}\\ [1ex] |T_{\beta}^{n}y-y_{0}|< e^{-S_ng(y)}     \end{array} \ {\text{for infinitely many}} \ n\in \N
\right\},
$$
where $S_nf(x)=\sum_{j=0}^{n-1}f(T_\beta^jx)$ is the Birkhoff sum. We calculate the Hausdorff dimension of this set and prove that it is the solution to some pressure function. This represents the first result of this kind for the higher dimensional beta dynamical systems.
\end{abstract}

\section{introduction}

The study of the Diophantine properties of the distribution of orbits for a measure preserving dynamical system has received much attention recently. Let $T:X\to X$ be a measure preserving transformation of the system $(X,\mathcal{B},\mu)$  with a consistent metric $d$.   If the transformation $T$ is ergodic with respect to the measure $\mu$, Poincare's recurrence theorem implies that, for almost every $x\in X$, the orbit $\{T^nx\}_{n=0}^\infty$ returns to an arbitrary but fixed neighbourhood of $x$ infinitely often. That is,  for any $x_0\in X$, for $\mu$-almost
all $x\in X,$ 
$$\liminf\limits_{n\rightarrow\infty} d(T^nx,x_0)=0.$$
Poincare's recurrence theorem is qualitative in nature but it does motivate the study of the distribution of  $T$-orbits of points in $X$ quantitatively. In other words, a natural motivation is to investigate \emph{how fast the above liminf tends to zero?} To this end, the spotlight is on the size of the set
$$D(T, \varphi):=\{x\in X: d(T^n x, x_0)<\varphi(n)~~\text{for infinitely many}~n\in \mathbb{N}\},$$
where $\varphi:\N\rightarrow \R_{\geq 0}$ is a positive function such that $\varphi(n)\rightarrow 0$ as $n\rightarrow \infty.$   The set $D(T,\varphi)$ can be viewed as the collection of points in $X$ whose $T$-orbit hits a shrinking target infinitely many times. The set $D(T,\varphi)$ is the dynamical analogue of the classical inhomogeneous well-approximable set
$$W(\varphi):=\{x\in [0, 1): |x-p/q-x_0|<\varphi(q)~~\text{for infinitely many}~p/q\in \mathbb Q\}.$$
As one would expect the `size' of both of these sets depend upon the nature of the function $\varphi$, that is, how fast it is approaching to zero. The size of the set $W(\varphi)$ in terms of Lebesgue measure or Hausdorff measure and dimension has been established even in the higher dimensional (linear form) settings, see \cite{BRV, WWX, HussainSimmons2} for further details.   In contrast,   not much is known for the higher dimensional version of the set $D(T, \varphi)$ { for general $T$}.   

\medskip

Following the work of Hill and Velani \cite{HillVelani1}, the Hausdorff dimension of the set $D(T, \varphi)$ has been  determined for many dynamical systems, from the system of rational expanding maps on their Julia sets to conformal iterated function systems \cite{Urbanski11}. We refer the reader to \cite{CHW} for a comprehensive discussion regarding the Hausdorff dimension of various dynamical systems.  In this paper, we confine {ourselves} to the two dimensional shrinking target problem in the beta dynamical system with a general error of approximation.


For a real number $\beta>1$, define the transformation $T_\beta:[0,1]\to[0,1]$ by $$T_\beta: x\mapsto \beta x\bmod 1.$$ This map generates the $\beta$-dynamical system $([0,1], T_\beta)$. It is well known that $\beta$-expansion is a typical example of an expanding non-finite Markov system whose properties are reflected by the orbit of some critical point, in other words, it is not a subshift of finite type with mixing properties. This causes difficulties in studying the metrical questions related to $\beta$-expansions.  General $\beta$-expansions have been widely studied in the literature, see for instance \cite{HLSW, TanWang, SeuretWang, LB, HussainWeiliang} and references therein. In particular, the Hausdorff dimension, { denoted throughout as $\dim_H$,  of  $D(T_\beta, \varphi)$ was obtained in \cite{LB}  and the Lebesgue measure and Hausdorff dimension of the set $$
D(T_{\beta}, \varphi_1,\varphi_2):=\left\{(x,y)\in [0,1]^2:  \begin{array}{ll} |T_{\beta}^{n}x-x_{0}|<\varphi_1(n)\\ [1ex] |T_{\beta}^{n}y-y_{0}|< \varphi_2(n)    \end{array} \ {\text{for infinitely many}} \ n\in \N
\right\},
$$
 was calculated in \cite{HussainWeiliang}. Here $x_0, y_0\in [0, 1]$ are fixed and the approximating functions $\varphi_1, \varphi_2$  are positive functions of $n$. 
\medskip

In 2014, Yann Bugeaud and Baowei Wang \cite{BugeaudWang} calculated the Hausdorff dimension of the set with the error of approximation given by the ergodic sum, i.e.  $$
E(T_\beta, h):=\left\{x\in [0,1]:  \begin{array}{ll} |T_{\beta}^{n}x-x_{0}|<e^{-S_nh(x)}  \end{array} \ {\text{for infinitely many}} \ n\in \N
\right\},
$$
where $h$ is a positive  continuous function on $[0,1]$ and $S_nh(x)=h(x)+\cdots+h(T_{\beta}^{n-1}x).$ {  Clearly the error of approximation is exponential depending upon the orbits $T_\beta x$.  Note that it is still an open problem whether $e^{-S_nh(x)}$ implies the arbitrary function $\varphi(n)$ or not. However, $e^{-S_nf(x)}$ reduces to  $\beta^{-n\tau}$ by considering $h(x)=\tau \log|T'(x)|$ for some $\tau>0$. Thus, it implies the  Jarn\'ik-Besicovitch type result for the set under consideration.}

\medskip

In this paper, we extend  Bugeaud and Wang's  set $E(T_\beta, h)$ to the two dimensional setting and calculate its Hausdorff dimension.}
Let $f,g$ be two positive  continuous function on $[0,1]$ and let $x_0, y_0\in [0, 1]$ be fixed.  Define
 $$
E(T_\beta, f,g):=\left\{(x,y)\in [0,1]^2:  \begin{array}{ll} |T_{\beta}^{n}x-x_{0}|<e^{-S_nf(x)}\\ [1ex] |T_{\beta}^{n}y-y_{0}|< e^{-S_ng(y)}     \end{array} \ {\text{for infinitely many}} \ n\in \N
\right\}.
$$
%

 The set $E(T_\beta, f,g)$ is the set of all points $(x, y)$ in the unit square such that the pair {$\{(T^{n}x, T^{n} y)\}$ is in the shrinking rectangle $B\left ( x_0,e^{-S_nf(x)} \right)\times B\left ( y_0,e^{-S_ng(y)} \right)$ for infinitely many $n$}. The rectangle  shrink to zero at exponential rates given by $e^{-S_nf(x)}$ and  {$e^{-S_ng(y)}$}. We shall prove the following result.

%
%
%
%
%
\begin{thm}\label{t2} Let $f,g$ be two continuous functions on $[0,1]$ with $f(x)\geq g(y)$ for all $x,y\in [0,1]$. Then
$$\hdim E(T_\beta, f,g)=\min\{s_1,s_2\},$$
where
\begin{align*}s_1&=\inf\{s\geq 0: P(f-s(\log\beta+f))+P(-g)\leq 0\}, \\ s_2&=\inf\{s\geq 0: P(-s(\log\beta+g))+\log\beta\leq 0\}.\end{align*}
\end{thm}
Here the notation  $P(\cdot)$ stands for the pressure function for the $\beta$-dynamical system associated to continuous potentials $f$ and $g$.  To keep the introductory section short, we formally give the definition of pressure function in section \ref{pre}. The reason that the Hausdorff dimension is in terms of the pressure function is because of the dynamical nature of the set $E(T_\beta, f, g)$. For the detailed analysis of the properties of the pressure function, ergodic sums for general dynamical systems we refer the reader to Chapter 9 of the book \cite{Walterbook}.

\medskip

The proof of this theorem splits into two parts: establishing the upper bound and then the lower bound. The upper bound is relatively easier to prove by using the definition of Hausdorff dimension on the natural cover of the set. However, establishing the lower bound is challenging and the main substance of this paper.  Actually, the main obstacle in determining the metrical properties of general $\beta$-expansions
lies in the difficulty of estimating the length of a general {cylinder} and, since we are dealing with two dimensional settings, as a consequence area of the cross product of general cylinders. As far as the Hausdorff dimension is concerned, one does not need to take all points into consideration; instead, one may choose a
subset of points with regular properties to approximate the set in question. This argument, in turn {requires}
some continuity of the dimensional number, when the system is approximated by its subsystem.

The paper is organised as follows. Section  \ref{pre} is devoted to recalling some elementary properties of $\beta$-expansions. Short proofs are also given when we could not find any reference. Definitions and some properties of the pressure function are stated in this section as well. In section \ref{upperbound}, we prove the upper bound of the Theorem \ref{t2}. In section \ref{lowerbound}, we prove the lower bound of Theorem \ref{t2} and since this carries the main weightage we subdivide this section into several subsections.


\section{Preliminaries}\label{pre}
We begin with a brief account on some basic properties of $\beta$-expansions and fixing some notation. We then state and prove two propositions which will give the covering and packing properties.

The $\beta$-expansion of real numbers was first introduced by R\'{e}nyi \cite{Renyi}, which is given by the following algorithm. For any $\beta>1$, let
\begin{equation}\label{e1}
T_{\beta}(0):=0,~~ T_{\beta}(x)=\beta x-\lfloor\beta x\rfloor, x\in[0,1),
\end{equation}
where $\lfloor\xi\rfloor$ is the integer part of $\xi\in \mathbb{R}$. By taking
$$\epsilon_{n}(x,\beta)=\lfloor\beta T_{\beta}^{n-1}x\rfloor\in \mathbb{N}$$
recursively for each $n\geq 1,$ every $x\in[0,1)$ can be uniquely expanded into a finite or an infinite sequence
\begin{equation*}\label{e2}
x=\frac{\epsilon_{1}(x,\beta)}{\beta}+\frac{\epsilon_{2}(x,\beta)}{\beta^2}+\cdots+\frac{\epsilon_{n}(x,\beta)}{\beta^n}+
\frac{T_{\beta}^n x}{\beta^n},
\end{equation*}
which is called the $\beta$-expansion of $x$ and the sequence  $\{\epsilon_{n}(x,\beta)\}_{n\geq1}$ is called the digit sequence of $x.$ We also write the $\beta$-expansion of $x$ as $$\epsilon(x,\beta)=\big(\epsilon_{1}(x,\beta),\cdots,\epsilon_{n}(x,\beta),\cdots\big).$$ The system $([0,1],T_{\beta})$ is called the \emph{$\beta$-dynamical system} or just the \emph{$\beta$-system.}

\begin{dfn}
A finite or an infinite sequence $(w_{1},w_{2},\cdots)$ is said to be admissible (with respect to the base $\beta$), if there exists an $x\in[0,1)$ such that the digit sequence  of $x$ equals $(w_{1},w_{2},\cdots).$
\end{dfn}

Denote by $\Sigma_{\beta}^n$ the collection of all admissible sequences of length $n$ and by $\Sigma_{\beta}$ that of all infinite admissible sequences.

Let us now turn to the infinite $\beta$-expansion of $1$, which plays an important role in the study of $\beta$-expansion. Applying algorithm $(\ref{e1})$ to the number $x=1$, then the number $1$ can be expanded into a series, denoted by
$$1=\frac{\epsilon_{1}(1,\beta)}{\beta}+\frac{\epsilon_{2}(1,\beta)}{\beta^2}+\cdots+\frac{\epsilon_{n}(1,\beta)}{\beta^n}+
\cdots.$$

If the above series is finite, i.e. there exists $m\geq1$ such that $\epsilon_{m}(1,\beta)\neq 0$ but $\epsilon_{n}(1,\beta)=0$ for $n>m$, then $\beta$ is called a simple Parry number. In this case,  we write $$\epsilon^*(1,\beta):=(\epsilon_{1}^{*}(\beta),\epsilon_{2}^{*}(\beta),\cdots)=(\epsilon_{1}(1,\beta),\cdots,\epsilon_{m-1}(1,\beta),
\epsilon_{m}(1,\beta)-1)^{\infty},$$
where $(w)^\infty$ denotes the periodic sequence $(w,w,w,\cdots).$ If $\beta$ is not a simple Parry number, we write
$$\epsilon^*(1,\beta):=(\epsilon_{1}^{*}(\beta),\epsilon_{2}^{*}(\beta),\cdots)=(\epsilon_{1}(1,\beta),\epsilon_{2}(1,\beta),
\cdots).$$

In both cases, the sequence $(\epsilon_{1}^{*}(\beta),\epsilon_{2}^{*}(\beta),\cdots)$ is called the infinite $\beta$-expansion of $1$ and we always have that
\begin{equation*}\label{e3}
1=\frac{\epsilon_{1}^*(\beta)}{\beta}+\frac{\epsilon_{2}^*(\beta)}{\beta^2}+\cdots+\frac{\epsilon_{n}^*(\beta)}{\beta^n}+
\cdots.
\end{equation*}

The lexicographical order $\prec$ between the infinite sequences is defined as follows:
$$w=(w_{1},w_{2},\cdots,w_{n},\cdots)\prec w'=(w_{1}',w_{2}',\cdots,w_{n}',\cdots)$$
if there exists $k\geq1$ such that $w_{j}=w_{j}'$ for $1\leq j<k$, while $w_{k}<w_{k}'.$ The notation $w\preceq w'$ means that $w\prec w'$ or $w=w'.$ This ordering can be extended to finite blocks by identifying a finite block $(w_{1},w_2,\cdots,w_n)$ with the sequence $(w_{1},w_2,\cdots,w_n,0,0,\cdots)$.

The following result due to Parry \cite{Parry} is a criterion for the admissibility of a sequence.
\begin{lem}[Parry \cite{Parry}]\label{le1} Let $\beta>1$ be a real number. Then a non-negative integer sequence $\epsilon=(\epsilon_1,\epsilon_2,\cdots)$ is admissible if and only if, for any $k\geq 1$,
    $$(\epsilon_k,\epsilon_{k+1},\cdots)\prec(\epsilon_1^*(\beta),\epsilon_2^*(\beta),\cdots).$$
\end{lem}

The following result of R\'enyi implies that the dynamical system $([0,1],T_\beta)$ admits $\log\beta$
as its topological entropy.

\begin{lem}[R\'enyi \cite{Renyi}]\label{le2}
Let $\beta>1.$ For any $n\geq 1,$ $$\beta^{n}\leq\#\Sigma_{\beta}^n\leq\frac{\beta^{n+1}}{\beta-1},$$ where $\#$ denotes the
   cardinality of a finite set.
\end{lem}

It is clear from this lemma that  $$\lim_{n\to\infty}\frac{\log\left(\#\Sigma_{\beta}^n\right)}{n}=\log\beta.$$
For any $(\epsilon_1,\cdots,\epsilon_n)\in \Sigma_{\beta}^n,$ call
$$I_{n}(\epsilon_1,\cdots,\epsilon_n):=\{x\in[0,1),\epsilon_j(x,\beta)=\epsilon_j,1\leq j\leq n\}$$
an $n$-th order cylinder $(\text{with respect to the base}~\beta)$. It is a left-closed and right-open interval with the left endpoint $$
\frac{\epsilon_1}{\beta}+\frac{\epsilon_2}{\beta^2}+\cdots+\frac{\epsilon_n}{\beta^n}
$$ and of length $$|I_n(\epsilon_1,\cdots,\epsilon_n)|\leq\frac{1}{\beta^n}.$$
Here and throughout the paper, we use $|\cdot|$ to denote the length of an interval. Note that the unit interval can be naturally partitioned into a disjoint union of cylinders; that is  for any $n\geq 1$,
\begin{equation*}\label{e4}[0,1]=\bigcup\limits_{(\epsilon_1,\cdots,\epsilon_n)\in\Sigma_{\beta}^n}
I_{n}(\epsilon_1,\cdots,\epsilon_n).
\end{equation*}

One difficulty in studying the metric properties of $\beta$-expansion is that the length of a cylinder is not regular. It may happen that $|I_n(\epsilon_1,\cdots, \epsilon_n)|\ll \beta^{-n}$. { Here $a\ll b$ is used to indicate that there exists a constant $c>0$ such that $a\leq cb$. We write $a\asymp b$ if $a\ll b\ll a$. } The following notation plays an important role to bypass this difficulty.

\begin{dfn}[Full cylinder] A cylinder $I_n(\epsilon_1,\cdots,\epsilon_n)$ is called full if it has maximal length, i.e. if $$|I_n(\epsilon_1,\cdots,\epsilon_n)|=\frac{1}{\beta^n}.$$ Correspondingly, we also call the word $(\epsilon_1,\cdots,\epsilon_n)$, defining the full cylinder $I_n(\epsilon_1,\cdots,\epsilon_n)$, a full word.
\end{dfn}
Next, we collect some properties about the distribution of full cylinders.
 \begin{pro}[Fan and Wang \cite{FanWang}]\label{le4}
An $n$-th order cylinder $I_{n}(\epsilon_{1}\cdots\epsilon_{n})$ is full, if and only if for any admissible sequence $(\epsilon_{1}',\epsilon_{2}',\cdots,\epsilon_{m}')\in\Sigma_{\beta}^m$ with $m\geq 1$, $$(\epsilon_{1}\cdots\epsilon_{n}, \epsilon_{1}',\epsilon_{2}',\cdots,\epsilon_{m}')\in\Sigma_{\beta}^{n+m}.$$
Moreover
      $$|I_{n+m}(\epsilon_{1},\cdots,\epsilon_{n},\epsilon_{1}',\cdots,\epsilon_{m}')|=
      |I_{n}(\epsilon_{1},\cdots,\epsilon_{n})|\cdot|I_{m}(\epsilon_{1}',\cdots,\epsilon_{m}')|.$$
So, for any two full cylinders $I_{n}(\epsilon_{1}\cdots\epsilon_{n}),~ I_{m}(\epsilon_{1}',\epsilon_{2}',\cdots,\epsilon_{m}')$, the cylinder $$I_{n+m}(\epsilon_{1},\cdots,\epsilon_{n},\epsilon_{1}',\cdots,\epsilon_{m}')$$ is also full.
\end{pro}

\begin{lem}[Bugeaud and Wang \cite{BugeaudWang}]\label{le3} For $n\geq 1$, among every $n+1$ consecutive cylinders of order $n$, there exists at least one full cylinder.
\end{lem}

As a consequence, one has the following relationship between balls and cylinders.
\begin{pro}[Covering property]\label{p1}Let $J$ be an interval of length ${\beta}^{-l}$ with $l\geq 1$. Then it can  be covered by at most $2(l+1)$ cylinders of order $l$.
\end{pro}
\begin{proof} By Lemma \ref{le3}, among any $2(l+1)$ consecutive cylinders of order $l$, there are at least $2$ full cylinders. So the total length of these intervals is larger than $2{\beta}^{-l}$. Thus $J$ can be covered by at most $2(l+1)$ cylinders of order $l$.
\end{proof}

The following result may have an independent interest.
\begin{pro}[Packing property]\label{p2} Fix $0<\epsilon<1$.  Let $~n_0$ be an integer such that $2n^2\beta<{\beta}^{(n-1)\epsilon}$ for all $n\ge n_0$. Let $J\subset [0,1]$ be an interval of length $r$ with $0<r<2n_{0}\beta^{-n_0}$. Then inside $J$, there exists a full cylinder $I_n$ satisfying
$$r\geq|I_n|>r^{1+\epsilon}.$$
\end{pro}
\begin{proof} Let $n>n_0$ be the integer such that $$2n\beta^{-n}\leq r<2(n-1)\beta^{-n+1}.$$
Since every cylinder of order $n$ is of length at most $\beta^{-n}$, the interval $J$ contains at least $2n-2\geq n+1$ consecutive cylinders of order $n$. Thus, by Lemma \ref{le3}, it contains a full cylinder of  order $n$ and we denote such a cylinder by  $I_n$. By the choice of $n_0$, we have
$$r\geq|I_n|=\beta^{-n}>\Big(2(n-1)(\beta^{-n+1})\Big)^{1+\epsilon}>r^{1+\epsilon}.$$
This completes the proof.
\end{proof}

Now we define a sequence of numbers $\beta_N$ approximating $\beta$ from below. For any $N$ with $\epsilon_N^*(\beta)\geq1,$ define $\beta_N$ to be the unique real solution to the algebraic equation
\begin{equation*}\label{g1}
1=\frac{\epsilon_{1}^*(\beta)}{\beta_N}+\frac{\epsilon_{2}^*(\beta)}{\beta_N^2}+\cdots+\frac{\epsilon_{N}^*(\beta)}{\beta_N^N}.
\end{equation*}
Then $\beta_N$ approximates $\beta$ frow below and the $\beta_N$-expansion of the unity is
$$(\epsilon_{1}^*(\beta),\cdots,\epsilon_{N-1}^*(\beta),\epsilon_{N}^*(\beta)-1)^{\infty}.$$

More importantly, by the criterion of admissible sequence, we have, for any $(\epsilon_1,\cdots,\epsilon_n)\in \Sigma_{\beta_N}^n$ and
$(\epsilon_1',\cdots,\epsilon_m')\in \Sigma_{\beta_N}^m$, that
\begin{equation}\label{g2}
(\epsilon_1,\cdots,\epsilon_n,0^N,\epsilon_1',\cdots,\epsilon_m')\in\Sigma_{\beta_N}^{n+N+m},
\end{equation}
where $0^N$ means a zero word of length $N$.

From the assertion $(\ref{g2})$, we get the following proposition.
\begin{pro}
For any $(\epsilon_1,\cdots,\epsilon_n)\in \Sigma_{\beta_N}^n$, $I_{n+N}(\epsilon_1,\cdots,\epsilon_n,0^N)$ is a full cylinder. So,
$$\frac{1}{\beta^{n+N}}\leq|I_{n}(\epsilon_1,\cdots,\epsilon_n)|\leq\frac{1}{\beta^{n}}.$$
\end{pro}


We end this section with a definition of the pressure function for $\beta$-dynamical system associated to some continuous potential $g$. The
readers are referred to \cite{Walters} for more details.
\begin{equation}\label{g3}
P(g, T_\beta):=\lim\limits_{n\rightarrow \infty}\frac{1}{n}\log\sum\limits_{(\epsilon_1,\cdots,\epsilon_n)\in\Sigma_\beta^n}
\sup\limits_{y\in I_n(\epsilon_1,\cdots,\epsilon_n)}e^{S_ng(y)},
\end{equation}
where $S_ng(y)$ denotes the ergodic sum $\sum_{j=0}^{n-1}g(T^j_\beta y)$. {Since $g$ is continuous, the limit does not depend upon the choice of $y$.} The existence of the limit \eqref{g3} follows from the subadditivity:
$$\log\sum\limits_{(\epsilon_1,\cdots,\epsilon_n, \epsilon_1^\prime,\cdots,\epsilon_m^\prime)\in\Sigma_\beta^{n+m}}
e^{S_{n+m}g(y)}\leq \log\sum\limits_{(\epsilon_1,\cdots,\epsilon_n)\in\Sigma_\beta^{n}}
e^{S_{n}g(y)}+\log\sum\limits_{(\epsilon_1^\prime,\cdots,\epsilon_m^\prime)\in\Sigma_\beta^{m}}
e^{S_{m}g(y)}.$$

\section{Proof of Theorem \ref{t2}: the upper bound}\label{upperbound}

As is typical in determining the Hausdorff dimension of a set; we split the proof of Theorem $\ref{t2}$ into two parts: the upper bound and the lower bound.

For any $U=(\epsilon_1,\cdots\epsilon_n)\in\Sigma_\beta^n$ and $W=(\omega_1,\cdots,\omega_n)\in\Sigma_\beta^n,$  we
always take $$x^*=\frac{\epsilon_{1}}{\beta}+\frac{\epsilon_{2}}{\beta^{2}}+\cdots+\frac{\epsilon_{n}}{\beta^{n}}$$ to be the left endpoint of $I_n(U)$ and $$y^*=\frac{\omega_{1}}{\beta}+\frac{\omega_{2}}{\beta^{2}}+\cdots+\frac{\omega_{n}}{\beta^{n}}$$ to be the left endpoint of $I_n(W)$.

%

Instead of directly considering the set $E(T_\beta,f,g)$, we will consider a closely related lim sup set
$$\overline{E}(T_\beta,f,g)=\bigcap_{N=1}^{\infty}\bigcup_{n=N}^{\infty}\bigcup_{U,W\in \Sigma_{\beta}^{n}}J_n(U)\times J_n(W),$$
where
\begin{align*}
 J_n(U)&=\{x\in[0,1]:|T_\beta^nx-x_0|<e^{-S_nf(x^*)}\},\\
J_n(W)&=\{y\in[0,1]:|T_\beta^ny-y_0|<e^{-S_ng(y^*)}\}.
 \end{align*}
In the sequel it will be clear that the set $\overline{E}(T_\beta,f,g)$ is easier to handle. Since $f$ and $g$ are continuous functions, for any $\delta>0$ and $n$ large enough, we have$$|S_nf(x)-S_nf(x^*)|<n\delta,\quad |S_ng(y)-S_ng(y^*)|<n\delta.$$
Thus we have $$\overline{E}(T_\beta, f+\delta,g+\delta)\subset E(T_\beta,f,g)\subset \overline{E}(T_\beta, f-\delta,g-\delta).$$
Therefore, to calculate the Hausdorff dimension of the set $E(T_\beta,f,g)$,  it is sufficient to determine the Hausdorff dimension of $\overline{E}(T_\beta,f,g)$.

The length of $J_n(U)$ satisfies $$|J_n(U)|\leq2\beta^{-n}e^{-S_nf(x^*)},$$
since, for every $x\in J_n(U)$, we have $$|x-(\frac{\epsilon_1}{\beta}+\cdots+\frac{\epsilon_n+x_0}{\beta^n})|=\frac{|T_\beta^nx-x_0|}{\beta^n}<\beta^{-n}e^{-S_nf(x^*)}.$$
Similarly, $$|J_n(W)|\leq2\beta^{-n}e^{-S_nf(y^*)}.$$
So, $\overline{E}(T_\beta,f,g)$ is a $\limsup$ set defined by a collection of rectangles. There are two ways to cover a single rectangle
$J_n(U)\times J_n(W)$ as follows.

\subsection{Covering by shorter side length}\label{sectionshorter} Recall that $f(x)\geq g(y)$ for all $x, y\in [0, 1]$. This implies that the length of $J_n(U)$ is shorter than the length of $J_n(W)$. Then the rectangle $J_n(U)\times J_n(W)$ can be covered by
 $$\frac{\beta^{-n}e^{-S_ng(y^*)}}{\beta^{-n}e^{-S_nf(x^*)}}=\frac{e^{S_nf(x^*)}}{e^{S_ng(y^*)}}$$balls of side length $\beta^{-n}e^{-S_nf(x^*)}.$

%

Since for each $N$, $$\overline{E}(T_\beta,f,g)\subseteq\bigcup_{n=N}^{\infty}\bigcup_{U,W\in \Sigma_{\beta}^{n}}J_n(U)\times J_n(W),$$
therefore, the $s$-dimensional Hausdorff measure $\H^s$ of $\overline{E}(T_\beta,f,g)$ can be estimated as
$$
\H^s\Big(\overline{E}(T_\beta,f,g)\Big)\le \liminf_{N\to\infty}\sum_{n=N}^{\infty}\sum_{U,W\in \Sigma_{\beta}^n}
\frac{e^{S_nf(x^*)}}{e^{S_ng(y^*)}}\Big(\frac{1}{\beta^ne^{S_nf(x^*)}}\Big)^s.
$$

Define $$s_1=\inf\{s\geq 0: P(f-s(\log\beta+f))+P(-g)\leq 0\}.$$
Then from the definition of the pressure function \eqref{g3}, it is clear that

$$P(f-s(\log\beta+f))+P(-g)\leq 0 \quad \iff \quad \sum_{n=1}^{\infty} \sum_{U,W\in \Sigma_{\beta}^n}
\frac{e^{S_nf(x^*)}}{e^{S_ng(y^*)}}\Big(\frac{1}{\beta^ne^{S_nf(x^*)}}\Big)^s<\infty.$$

Hence, for any $s>s_1$

$$\H^s\Big(\overline{E}(T_\beta,f,g)\Big)= 0.$$ Hence it follows that
$\hdim (\overline{E}(T_\beta,f,g))\leq s_1.$

%
%
%

\subsection{Covering by longer side length}\label{sectionlonger} From the previous subsection (\S\ref{sectionshorter}), it is clear that only one ball of side length  $\beta^{-n}e^{-S_ng(y^*)}$ is needed to cover the rectangle $J_n(U)\times J_n(W)$. Hence, in this case, the $s$-dimensional Hausdorff measure $\H^s$ of $\overline{E}(T_\beta,f,g)$ can be estimated as

$$\mathcal{H}^s(\overline{E}(T_\beta,f,g))\leq\liminf\limits_{N\rightarrow\infty}\sum_{n=N}^{\infty}\sum_{U,W\in \Sigma_{\beta}^n}\Big(
\frac{1}{\beta^{n}e^{S_ng(y^*)}}\Big)^s.
$$
Define $$ s_2=\inf\{s\geq 0: P(-s(\log\beta+g))+\log\beta\leq 0\}.$$Then,  from the definition of pressure function and Hausdorff measure, it follows that, for any $s>s_2$,
$\H^s\Big(\overline{E}(T_\beta,f,g)\Big)= 0.$ Hence,
$$\hdim (\overline{E}(T_\beta,f,g))\leq s_2.$$

%
%

\section{Theorem \ref{t2}: The lower bound }\label{lowerbound}
It should be clear from the previous section that proving the upper bound requires only a suitable covering of the set  $\overline{E}(T_\beta,f,g)$. However, in contrast, proving the lower bound is a challenging task, requiring all possible coverings to be considered  and, therefore, represents the main problem in metric Diophantine approximation (in various settings).  The following principle commonly known as the Mass Distribution Principle  \cite{Falconer_book}  has been used frequently for this purpose.

\begin{pro}[Falconer \cite{Falconer_book}]\label{mdp}  Let $E$ be a Borel measurable set in { $\mathbb{R}^d$} and $\mu$ be a Borel measure with $\mu(E)>0$. Assume that there exist two positive constant $c,\delta$ such that, for any set $U$ with diameter $|U|$ less than $\delta$, $\mu(U)\leq c |U|^s$, then $\hdim E\geq s$.
\end{pro}

Specifically, the mass distribution principle replaces the consideration of all coverings by the construction of a particular measure $\mu$ and it is typically deployed in two steps:

\begin{itemize}
\item construct a suitable Cantor subset $\mathcal F_\infty$ of $\overline{E}(T_\beta,f,g)$ and a probability measure $\mu$ supported on $\mathcal F_\infty$,

\item  show that for any fixed $c>0$, $\mu$  satisfies the condition that for any measurable set $U$ of sufficiently small diameter,  $\mu (U)\leq c |U|^s$.
 \end{itemize}
If this can be done, then by the mass distribution principle, it follows that  $$\hdim(\overline{E}(T_\beta,f,g))\geq \hdim(\mathcal F_\infty)\geq s.$$

The main intricate and substantive part  of this entire process is the construction of a suitable Cantor type subset $\mathcal F_\infty$ which supports a probability measure $\mu$.  In the remainder of this paper, we will construct a suitable Cantor type subset of the set $\overline{E}(T_\beta,f,g)$ and demonstrate that it satisfies the mass distribution principle.

\subsection*{Construction of the Cantor subset.}  We construct the Cantor subset $\mathcal F_\infty$ iteratively. Start by fixing an $\epsilon>0$ and assume that $f(x)\geq(1+\epsilon)g(y)\geq g(y)$ for all $x,y\in[0,1].$ We construct a Cantor subset level  by level and note that each level depends on its predecessor. Choose a  rapidly increasing subsequence $\{m_k\}_{k\geq 1}$ of positive integers with $m_1$ large enough.

\subsection{Level 1 of the Cantor set.} Let $n_1=m_1.$ For any $U_1,W_1\in \Sigma_{\beta_N}^{n_1}$ ending with the zero word of order $N$, i.e.  $0^N.$  Let $x_1^*\in I_{n_1}(U_1), ~y_1^*\in I_{n_1}(W_1).$ From Proposition \ref{p2}, it follows that there are two full
cylinders $I_{k_1}(K_1), I_{l_1}(L_1)$ such that
\begin{align*}I_{k_1}(K_1)&\subset B\Big(x_0,e^{-S_{n_1}f(x_1^*)}\Big), \\ I_{l_1}(L_1)& \subset B\Big(y_0,e^{-S_{n_1}g(y_1^*)}\Big), \end{align*}
and $$e^{-S_{n_1}f(x_1^*)}>\beta^{-k_1}>\Big(e^{-S_{n_1}f(x_1^*)}\Big)^{1+\epsilon},$$
$$ e^{-S_{n_1}g(y_1^*)}>\beta^{-l_1}>\Big(e^{-S_{n_1}g(y_1^*)}\Big)^{1+\epsilon}=e^{-S_{n_1}(1+\epsilon)g(y_1^*)}.$$

\medskip

So, we get a subset $I_{n_1+k_1}(U_1,K_1)\times I_{n_1+l_1}(W_1,L_1)$ of $J_{n_1}(U_1)\times J_{n_1}(W_1)$. Since $f(x)\geq (1+\epsilon)g(y)$ for all $x,y\in [0,1],$ then $k_1\geq l_1.$ It should be noted that $K_1$ and $L_1$  depends on $U_1$ and $W_1$ respectively. Consequently, for different $U_1$ and $W_1,$ the choice of  $K_1$ and $L_1$ may be different.

The first level of the Cantor set is defined as
$$\mathcal{F}_1=\Big\{I_{n_1+k_1}(U_1,K_1)\times I_{n_1+l_1}(W_1,L_1):U_1,W_1\in \Sigma_{\beta_N}^{n_1} ~~\text{ending with}~~ 0^N\Big\},$$
which is composed of a collection of rectangles. Next, we cut each rectangle into balls with the radius as the shorter side length of the rectangle:
\begin{align*}
 I_{n_1+k_1}(U_1,K_1)\times I_{n_1+l_1}(W_1,L_1)\rightarrow\Big\{I_{n_1+k_1}&(U_1,K_1)\times I_{n_1+k_1}(W_1,L_1,H_1):H_1\in \Sigma_{\beta_N}^{k_1-l_1}\Big\}.
\end{align*}
Then we get a collection of balls
\begin{align*}
\mathcal{G}_1=\Big\{I_{n_1+k_1}(U_1,K_1)\times I_{n_1+k_1}(W_1,L_1,H_1):U_1,&W_1\in \Sigma_{\beta_N}^{n_1} \  \text{ending with}~~ 0^N,
H_1\in \Sigma_{\beta_N}^{k_1-l_1}\Big\}.
\end{align*}

\subsection{Level 2 of the Cantor set.} Fix a $J_1=I_{n_1+k_1}(\Gamma_1)\times I_{n_1+k_1}(\Upsilon_1)$ in $\mathcal{G}_1$. We define the local sublevel $\mathcal{F}_2(J_1)$ as follows.

Choose a large integer $m_2$ such that
$$\frac{\epsilon}{1+\epsilon}\cdot m_2\log\beta\geq\Big(n_1+\sup\{k_1:I_{n_1+k_1}(\Gamma_1)\}\Big)||f||,$$
where $||f||=\sup\Big\{|f(x)|:x\in[0,1]\Big\}.$

Write $n_2=n_1+k_1+m_2.$ Just like the first level of the Cantor set, for any $U_2,W_2\in \Sigma_{\beta_N}^{m_2}$ ending with $0^N$, applying Proposition \ref{p2} to $J_{n_2}(\Gamma_1,U_2)\times J_{n_2}(\Upsilon_1,W_2)$, we can get two full cylinders $I_{k_2}(K_2)$,~$I_{l_2}(L_2)$ such that

$$I_{k_2}(K_2)\subset B\Big(x_0,e^{-S_{n_2}f(x_2^*)}\Big),~ I_{l_2}(L_2)\subset B\Big(y_0,e^{-S_{n_2}g(y_2^*)}\Big)$$
and $$e^{-S_{n_2}f(x_2^*)}>\beta^{-k_2}>\Big(e^{-S_{n_2}f(x_2^*)}\Big)^{1+\epsilon},$$
$$ e^{-S_{n_2}g(y_2^*)}>\beta^{-l_2}>\Big(e^{-S_{n_2}g(y_2^*)}\Big)^{1+\epsilon}=e^{-S_{n_2}(1+\epsilon)g(y_2^*)},$$
where $x_2^*\in I_{n_2}(\Gamma_1,U_2),~y_2^*\in I_{n_2}(\Upsilon_1,W_2).$

Obviously, we get a subset $$I_{n_2+k_2}(\Gamma_1,U_2,K_2)\times I_{n_2+l_2}(\Upsilon_1,W_2,L_2)$$ of $$J_{n_2}(\Gamma_1,U_2)\times J_{n_2}(\Upsilon_1,W_2)$$ and $k_2\geq l_2$.

Then, the second level of the Cantor set is defined as

\begin{align*}
\mathcal{F}_2(J_1)=\Big\{I_{n_2+k_2}(\Gamma_1,U_2,K_2)\times &I_{n_2+l_2}(\Upsilon_1,W_2,L_2):U_2,W_2\in \Sigma_{\beta_N}^{m_2} ~~\text{ending with}~~ 0^N\Big\},
\end{align*}
which is composed of a collection of rectangles.

Next, we cut each rectangle into balls with the radius as the shorter sidelength of the rectangle:
\begin{align*}
I_{n_2+k_2}(\Gamma_1,U_2,K_2)\times &I_{n_2+l_2}(\Upsilon_1,W_2,L_2)\rightarrow\Big\{I_{n_2+k_2}(\Gamma_1,U_2,K_2)\\&\times I_{n_2+k_2}(\Upsilon_1,W_2,L_2,H_2):H_2\in \Sigma_{\beta_N}^{k_2-l_2}\Big\}:=\mathcal{G}_2(J_1).
\end{align*}

Therefore, the second level is defined as
$$\mathcal{F}_2=\bigcup
\limits_{J\in \mathcal{G}_1}\mathcal{F}_2(J),~\mathcal{G}_2=\bigcup\limits_{J\in \mathcal{G}_1}\mathcal{G}_2(J).$$

\subsection{From Level $(i-1)$ to Level $i$.} Assume that the $(i-1)$th level of the Cantor set $\mathcal{G}_{i-1}$ has been defined. Let
$J_{i-1}=I_{n_{i-1}+k_{i-1}}(\Gamma_{i-1})\times I_{n_{i-1}+k_{i-1}}(\Upsilon_{i-1})$ be a generic element in $\mathcal{G}_{i-1}.$  We define the
local sublevel $\mathcal{F}_i(J_{i-1})$ as follows.

Choose a large integer $m_{i}$ such that
\begin{equation}\label{espil}
\frac{\epsilon}{1+\epsilon}\cdot m_i\log\beta\geq\Big(n_{i-1}+\sup\big\{k_{i-1}:I_{n_{i-1}+k_{i-1}}(\Gamma_{i-1})\big\}\Big)||f||.
\end{equation}

Write $n_i=n_{i-1}+k_{i-1}+m_i.$ For each $U_i,W_i\in \Sigma_{\beta_N}^{m_i}$ ending with $0^N$, apply Proposition \ref{p2} to $$J_{n_i}(\Gamma_{n_{i-1}+k_{i-1}},U_{i})\times J_{n_i}(\Upsilon_{n_{i-1}+k_{i-1}},W_i),$$ we can get two full cylinders $I_{k_i}(K_i)$,~$I_{l_i}(L_i)$ such that

$$I_{k_i}(K_i)\subset B\Big(x_0,e^{-S_{n_i}f(x_i^*)}\Big), I_{l_i}(L_i)\subset B\Big(y_0,e^{-S_{n_i}g(y_i^*)}\Big)$$
and $$e^{-S_{n_i}f(x_i^*)}>\beta^{-k_i}>\Big(e^{-S_{n_i}f(x_i^*)}\Big)^{1+\epsilon},$$
$$ e^{-S_{n_i}g(y_i^*)}>\beta^{-l_i}>\Big(e^{-S_{n_i}g(y_i^*)}\Big)^{1+\epsilon}=e^{-S_{n_i}(1+\epsilon)g(y_i^*)},$$
where $x_i^*\in I_{n_i}(\Gamma_{i-1},U_{i}), ~y_i^*\in I_{n_i}(\Upsilon_{i-1},W_{i}).$

Obviously, we get a subset $$I_{n_i+k_i}(\Gamma_{i-1},U_i,K_i)\times I_{n_i+l_i}(\Upsilon_{i-1},W_i,L_i)$$ of $$J_{n_i}(\Gamma_{i-1},U_i)\times J_{n_i}(\Upsilon_{i-1},W_i)$$ and $k_i\geq l_i$.
Then, the $i$-th level of the Cantor set is defined as

\begin{align*}
\mathcal{F}_i(J_{i-1})=\Big\{I_{n_i+k_i}(\Gamma_{i-1},U_i,K_i)\times &I_{n_i+l_i}(\Upsilon_{i-1},W_i,L_i):U_i,W_i\in \Sigma_{\beta_N}^{m_i} ~~\text{ending with}~~ 0^N\Big\},
\end{align*}
which is composed of a collection of rectangles.
As before, we cut each rectangle into balls with the radius as the shorter sidelength of the rectangle:
\begin{align*}
I_{n_i+k_i}(\Gamma_{i-1},U_i,K_i)\times I_{n_i+l_i}&(\Upsilon_{i-1},W_i,L_i)\rightarrow\Big\{I_{n_i+k_i}(\Gamma_{i-1},U_i,K_i)\times\\ &I_{n_i+k_i}(\Upsilon_{i-1},W_i,L_i,H_i):H_i\in \Sigma_{\beta_N}^{k_i-l_i}\Big\}:=\mathcal{G}_i(J_{i-1}).
\end{align*}

Therefore, the $i$-th level is defined as
$$\mathcal{F}_i=\bigcup
\limits_{J\in \mathcal{G}_{i-1}}\mathcal{F}_i(J),~\mathcal{G}_i=\bigcup\limits_{J\in \mathcal{G}_{i-1}}\mathcal{G}_i(J).$$

Finally, the Cantor set is defined as
$$\mathcal{F}_\infty=\bigcap\limits_{i=1}^\infty\bigcup\limits_{J\in \mathcal{F}_i}J=\bigcap\limits_{i=1}^\infty\bigcup\limits_{I\in \mathcal{G}_i}I.$$
It is straightforward to see that  $\mathcal{F}_\infty\subset \overline{E}(T_\beta,f,g)$.

\begin{rem}\label{rem1}It should be noted that the integer $k_i$ depends upon $\Gamma_{i-1}$ and $U_i$. However,(assume that $f$ is strictly positive,
otherwise replace $f$ by $f+\epsilon$ ), since $m_i$ can be chosen such that $m_i\gg n_{i-1}$ for all $n_{i-1}$. So,
$$\beta^{-k_i}\asymp e^{-S_{n_i}f(x_i^*)}=\Big(e^{-S_{m_i}f(T_\beta^{n_{i-1}+k_{i-1}}x_i^*)}\Big)^{1+\epsilon}.$$
where $x_i^*\in I_{n_{i-1}+k_{i-1}+m_i}(\Gamma_{i-1},U_i).$
In other words, $k_i$ is almost dependent only on $U_i$ and
\begin{equation}\label{eqki}
\beta^{-k_i}\asymp e^{-S_{m_i}f(x_i')}, {x_i'}\in I_{m_i}(U_i).
\end{equation}

The same is true for $l_i$,
\begin{equation}\label{eqli}
\beta^{-l_i}\asymp e^{-S_{m_i}f(y_i')}, {y_i'}\in I_{m_i}(W_i).
\end{equation}

\end{rem}

\subsection{\bf Supporting measure }
Now we  construct a probability measure $\mu$ supported on $\mathcal{F}_\infty$, which is defined by distributing masses among the cylinders with non-empty intersection with $\mathcal{F}_\infty$. The process splits into two cases: when $s_0>1$ and $ 0\leq s_0\leq 1$.

\subsubsection*{\bf  Case I:  $s_0>1$}

In this case, for any $1<s<s_0,$ notice that
$$\frac{e^{S_nf(x')}}{e^{S_ng(y')}}\left(\frac{1}{\beta^ne^{S_nf(x')}}\right)^s\leq\left(\frac{1}{\beta^ne^{S_ng(y')}}\right)^s.$$
This means that the covering the rectangle $J_n(U)\times J_n(W)$ by balls of shorter side length preferable  and therefore, it reasonable to define the probability measure on smaller balls. To this end, let $s_i$ be the solution to the equation
$$\sum\limits_{U,W\in \Sigma_{\beta_N}^{m_i}}
\frac{e^{S_{m_i}f(x_i')}}{e^{S_{m_i}g(y_i')}}\Big(\frac{1}{\beta^{m_i}e^{S_{m_i}f(x_i')}}\Big)^s=1,$$
where $x_i'\in I_{m_i}(U_i),~ y_i'\in I_{m_i}(W_i).$

By the continuity of the pressure function $P(T_\beta,f)$ with respect to $\beta$ ~\cite[Theorem 4.1]{TanWang}, it can be shown that $s_i\rightarrow s_0$ when $m_i\rightarrow \infty.$ Thus without loss of generality, we choose that all $m_i$ are large enough such that $s_i>1$ for all $i$ and
$|s_i-s_0|=o(1).$

We systematically define the measure $\mu$ on the Cantor set by defining it on the basic cylinders first. Recall that for the level 1 of the Cantor set construction, we assumed that $n_1=m_1.$ For sub-levels of the Cantor set, roughly speaking, the role of $m_1$ and $m_k$ are to denote how many positions where the digits can be chosen (almost) freely. While $n_1$ and $n_k$ denote the length of a word in level $\F_k$ before shrinking.

\begin{itemize}

\item Let $I_{n_1+k_1}(U_1,K_1)\times I_{n_1+k_1}(W_1,L_1,H_1)$ be a generic cylinder in $\mathcal{G}_1.$ Then define
$$\mu\Big(I_{n_1+k_1}(U_1,K_1)\times I_{n_1+k_1}(W_1,L_1,H_1)\Big)=\Big(\frac{1}{\beta^{m_1}e^{S_{m_1}f(x_1')}}\Big)^{s_1},$$
where $ x_1'\in I_{m_1}(U_1)$.
\end{itemize}
Assume that the measure on the cylinders of order $(i-1)$ has been well define. To define measure on the $i$th cylinder,

\begin{itemize}
\item Let $I_{n_i+k_i}(\Gamma_{i-1},U_i,K_i)\times I_{n_i+k_i}(\Upsilon_{i-1},W_i,L_i,H_i)$ be a generic
 $i$th cylinder in $\mathcal{G}_i$. Define the probability measure $\mu$  as
\begin{align*}
\mu\Big(I_{n_i+k_i}(\Gamma_{i-1},U_i,K_i)&\times I_{n_i+k_i}(\Upsilon_{i-1},W_i,L_i,H_i)\Big)= \\&\mu\Big(I_{{n_{i-1}+k_{i-1}}}(\Gamma_{i-1})\times I_{{n_{i-1}+k_{i-1}}}(\Upsilon_{i-1})\Big)\times\Big(\frac{1}{\beta^{m_i}e^{S_{m_i}f(x_i')}}\Big)^{s_i},
\end{align*}
\end{itemize}
where $ x_i'\in I_{m_i}(U_i)$.

The measure of a rectangle in $\mathcal{F}_i$ is then given as
\begin{align*}
 &\mu\Big(I_{n_i+k_i}(\Gamma_{i-1},U_i,K_i)\times I_{n_i+l_i}(\Upsilon_{i-1},W_i,L_i)\Big)\\
 &=\mu\Big(I_{{n_{i-1}+k_{i-1}}}(\Gamma_{i-1})\times I_{{n_{i-1}+k_{i-1}}}(\Upsilon_{i-1})\Big)\times{\#\Sigma_\beta^{k_i-l_i}}\times\Big(\frac{1}{\beta^{m_i}e^{S_{m_i}f(x_i')}}\Big)^{s_i}\\
 &\asymp \mu\Big(I_{{n_{i-1}+k_{i-1}}}(\Gamma_{i-1})\times I_{{n_{i-1}+k_{i-1}}}(\Upsilon_{i-1})\Big)\times\frac{e^{S_{m_i}f(x_i')}}{e^{S_{m_i}g(y_i')}}\times\Big(\frac{1}{\beta^{m_i}e^{S_{m_i}f(x_i')}}\Big)^{s_i},
\end{align*}
where the last inequality follows from the estimates \eqref{eqki}  and \eqref{eqli}.

\subsubsection{Estimation of the $\mu$-measure of cylinders.}
For any $i\geq 1$ consider the generic cylinder,
$$I:=I_{n_i+k_i}(\Gamma_{i-1},U_i,K_i)\times I_{n_i+k_i}(\Upsilon_{i-1},W_i,L_i,H_i).$$
We would like to show by induction that, for any $1<s< s_0$,  $$\mu(I)\leq|I|^{s/(1+\epsilon)}.$$

When $i=1$. The length of $I$ is given as $$|I|=\beta^{-m_1-k_1}\geq\beta^{-m_1}\cdot\Big(e^{-S_{n_1}f(x_1^*)}\Big)^{1+\epsilon}=\beta^{-m_1}\cdot\Big(e^{-S_{m_1}f(x_1^*)}\Big)^{1+\epsilon}.$$
But, by the definition of the measure $\mu$, it is clear that
$$\mu(I)\leq|I|^{s_1}\leq|I|^{s/(1+\epsilon)}.$$

Now we consider the inductive process. Assume that
\begin{align*}
\mu(I_{n_{i-1}+k_{i-1}}(\Gamma_{i-1})\times &I_{n_{i-1}+k_{i-1}}(\Upsilon_{i-1}))
\leq|I_{n_{i-1}+k_{i-1}}(\Gamma_{i-1})\times I_{n_{i-1}+k_{i-1}}(\Upsilon_{i-1})|^{s/(1+\epsilon)}.
\end{align*}

Let $$I=I_{n_i+k_i}(\Gamma_{i-1},U_i,K_i)\times I_{n_i+k_i}(\Upsilon_{i-1},W_i,L_i,H_i)$$ be a generic cylinder in $\mathcal{G}_i$. One one hand, its length satisfies
\begin{align*}\label{f9}
|I|&=\beta^{-n_i-k_i}=|I_{n_{i-1}+k_{i-1}}(\Gamma_{i-1})\times I_{n_{i-1}+k_{i-1}}(\Upsilon_{i-1})|\times\beta^{-m_i}\times\beta^{-k_i}\nonumber\\
&\geq|I_{n_{i-1}+k_{i-1}}(\Gamma_{i-1})\times I_{n_{i-1}+k_{i-1}}(\Upsilon_{i-1})|\times\beta^{-m_i}\Big(e^{-S_{n_i}f(x_i^*)}\Big)^{1+\epsilon},
\end{align*}
where $x_i^*\in I_{n_i}(\Gamma_i,U_i).$

We compare $S_{n_i}f(x_i^*)$ and $S_{m_i}f(x_i')$ , by (\ref{espil}) we have
\begin{align*}|S_{n_i}f(x_i^*)-S_{m_i}f(x_i')|&=|S_{n_{i-1}+k_{i-1}}f(x_i^*)|\\ &\leq({n_{i-1}+k_{i-1}})\|f\| \\ &\leq \frac{\epsilon}{1+\epsilon}m_i\log\beta, \end{align*}
where  $x_i'\in I_{m_i}(U_i).$
So, we get

\begin{equation}\label{geq}
|I|\geq|I_{n_{i-1}+k_{i-1}}(\Gamma_{i-1})\times I_{n_{i-1}+k_{i-1}}(\Upsilon_{i-1})|\times\Big(\beta^{-m_i}e^{-S_{m_i}f(x_i')}\Big)^{1+\epsilon}.
\end{equation}

On the other hand, by the definition of the measure $\mu$ and the induction, we have that
\begin{align*}
\mu(I)&=\mu\Big(I_{{n_{i-1}+k_{i-1}}}(\Gamma_{i-1})\times I_{{n_{i-1}+k_{i-1}}}(\Upsilon_{i-1})\Big)\times\Big(\beta^{-m_i}e^{-S_{m_i}f(x_i')}\Big)^{s_i}\\
&\leq|I_{n_{i-1}+k_{i-1}}(\Gamma_{i-1})\times I_{n_{i-1}+k_{i-1}}(\Upsilon_{i-1})|^{s/(1+\epsilon)}
\Big((\beta^{-m_i}e^{-S_{m_i}f(x_i')})^{1+\epsilon}\Big)^{s/(1+\epsilon)}\\
&\leq |I|^{s/(1+\epsilon)}.
\end{align*}

In the following steps, for any $(x,y)\in \mathcal{F}_\infty,$ we will estimate the measure of $I_n(x)\times I_n(y)$ compared with its length $\beta^{-n}.$ By the construction of $\mathcal{F}_\infty,$ there exists $\{k_i,l_i\}_{i\geq1}$ such that for all $i\geq 1,$
$$(x,y)\in I_{n_i+k_i}(\Gamma_{i-1},U_i,K_i)\times I_{n_i+l_i}(\Upsilon_{i-1},W_i,L_i).$$

We remark that though $\{k_i,l_i\}$ are different for different cylinders composing $\mathcal{F}_\infty$ is given, once $(x,y)\in \mathcal{F}_\infty$ is given, the corresponding integers
$\{k_i,l_i\}$ are fixed.

For any $n\geq 1$, Let $i\geq1$ be the integer such that
$$n_{i-1}+k_{i-1}<n\leq n_i+k_i=n_{i-1}+k_{i-1}+m_i+k_i.$$

\noindent{\bf Step 1.} When $n_{i-1}+k_{i-1}+m_i+l_i\leq n \leq n_{i}+k_{i}=n_{i-1}+k_{i-1}+m_i+k_i.$

Then the cylinder $I_n(x)\times I_n(y)$ contains $\beta^{n_{i}+k_{i}-n}$ cylinders in $\mathcal{G}_i$ with order $n_{i}+k_{i}$. Note that by the definition of
$\{k_j,l_j\}_{1\leq j\leq i}$, the first $i$-pairs $\{k_j,l_j\}_{1\leq j\leq i}$ depends only on the first $n_i$ digits of $(x,y)$. So the measure of the sub-cylinder of order $n_{i}+k_{i}$
are the same. So, its measure of $I_n(x)\times I_n(y)$ can be estimated as
\begin{align*}
\mu\Big(I_n(x)\times I_n(y)\Big)
=\mu\Big(I_{{n_{i-1}+k_{i-1}}}(\Gamma_{i-1})&\times I_{{n_{i-1}+k_{i-1}}}(\Upsilon_{i-1})\Big) \times \Big(\beta^{-m_i}e^{-S_{m_i}f(x_i')}\Big)^{s_i}\times \beta^{n_{i}+k_{i}-n}.
\end{align*}
Thus by the measure estimation of cylinders of order $n_{i-1}+k_{i-1}$ and the choice of $k_i$, one has that
\begin{align*}
\mu\Big(I_n(x)\times I_n(y)\Big)&\leq \Big(\beta^{-n_{i-1}-k_{i-1}}\Big)^{s/(1+\epsilon)}\Big(\beta^{-m_{i}-k_{i}}\Big)^{s/(1+\epsilon)}\times\beta^{n_{i}+k_{i}-n}\\
&=\Big(\beta^{-n_{i}-k_{i}}\Big)^{s/(1+\epsilon)}\times\beta^{n_{i}+k_{i}-n}
\\ &\leq \Big(\beta^{-n}\Big)^{s/(1+\epsilon)},
\end{align*}by noting that $n \leq n_{i}+k_{i}$ and ${s/(1+\epsilon)}>1.$


\noindent{\bf Step 2.}  When $n_{i-1}+k_{i-1}+m_i\leq n \leq n_{i}+l_{i}=n_{i-1}+k_{i-1}+m_i+l_i.$

Recalling the definition of $n_{i}+k_{i}$, the first $i$-pairs $\{k_j,l_j\}_{1\leq j\leq i}$ depends only on the first $n_i$ digits of $(x,y)$. So the measure of the sub-cylinder in $\mathcal{G}_i$ with  order $n_{i}+k_{i}$
are the same. It is clear that the cylinder $I_n(x)\times I_n(y)$ contains $\beta^{k_{i}-l_i}$ cylinders of order $n_{i}+k_{i}$.
So, its measure of $I_n(x)\times I_n(y)$ can be estimated as
\begin{align*}
\mu\Big(I_n(x)\times I_n(y)\Big)
=\mu\Big(I_{{n_{i-1}+k_{i-1}}}(\Gamma_{i-1})&\times I_{{n_{i-1}+k_{i-1}}}(\Upsilon_{i-1})\Big) \times \Big(\beta^{-m_i}e^{-S_{m_i}f(x_i'^*)}\Big)^{s_i}\times \beta^{k_{i}-l_i}.
\end{align*}
Thus by the measure estimation of cylinders of order $n_{i-1}+k_{i-1}$ and the choice of $k_i$, one has that
\begin{align*}
\mu\Big(I_n(x)\times I_n(y)\Big)&\leq \Big(\beta^{-n_{i-1}-k_{i-1}}\Big)^{s/(1+\epsilon)}\Big(\beta^{-m_{i}-k_{i}}\Big)^{s/(1+\epsilon)}\times\beta^{k_{i}-l_i}\\
&=\Big(\beta^{-n_{i}-k_{i}}\Big)^{s/(1+\epsilon)}\times\beta^{k_{i}-l_i}\\
&\leq\Big(\beta^{-n_i-l_i}\Big)^{s/(1+\epsilon)}
\\ &\leq \Big(\beta^{-n}\Big)^{s/(1+\epsilon)},
\end{align*}
by noting that $n \leq n_{i}+l_{i}$ and ${s/(1+\epsilon)}>1.$

\noindent{\bf Step 3.}  When $n_{i-1}+k_{i-1}\leq n \leq n_{i-1}+k_{i-1}+m_i.$

Assume that $U_i=(\epsilon_1,\epsilon_2,\ldots,\epsilon_{m_i}), W_i=(\omega_1,\omega_2,\ldots,\omega_{m_i}).$ Denote $l=n-(n_{i-1}+k_{i-1})$ and
$h=m_i-l$. Then
\begin{align*}
&\mu\left(I_n(x)\times I_n(y)\right)\\ &=\sum\limits_{\substack{(\epsilon_{l+1},\ldots,\epsilon_{m_i})\in \Sigma_\beta^l\\ (\omega_{l+1},\ldots,\omega_{m_i})\in \Sigma_\beta^h}}
\mu\left(I_{n_i+k_i}(\Gamma_{i-1},U_i,K_i)
\times I_{n_i+k_i}(\Upsilon_{i-1},W_i,L_i,H_i)\right)\times\beta^{k_{i}-l_i}\\
&=\mu\left(I_{{n_{i-1}+k_{i-1}}}(\Gamma_{i-1})\times I_{{n_{i-1}+k_{i-1}}}(\Upsilon_{i-1})\right)\times
\sum\limits_{\substack{(\epsilon_{l+1},\ldots,\epsilon_{m_i})\in \Sigma_\beta^l \\ (\omega_{l+1},\ldots,\omega_{m_i})\in \Sigma_\beta^h}}\left(\beta^{-m_i}e^{-S_{m_i}f(x_i')}\right)^{s_i}
\times\beta^{k_{i}-l_i}\\
&=\mu\left(I_{{n_{i-1}+k_{i-1}}}(\Gamma_{i-1})\times I_{{n_{i-1}+k_{i-1}}}(\Upsilon_{i-1})\right)\times
\sum\limits_{\substack{(\epsilon_{l+1},\ldots,\epsilon_{m_i})\in \Sigma_\beta^l\\ (\omega_{l+1},\ldots,\omega_{m_i})\in \Sigma_\beta^h}}\frac{e^{S_{m_i}f(x_i')}}{e^{S_{m_i}g(y_i')}}
\left(\beta^{-m_i}e^{-S_{m_i}f(x_i')}\right)^{s_i}.
\end{align*}
Then by the estimation on the measure of cylinders of order $n_{i-1}+k_{i-1}$ and let $(\widetilde{x_i'},\widetilde{y_i'})=(T_\beta^lx_i',T_\beta^ly_i')$, we get
\begin{align*}
\mu\left(I_n(x)\times I_n(y)\right)&\leq(\beta^{-n_{i-1}-k_{i-1}})^{s/(1+\epsilon)}\cdot\frac{e^{S_{l}f(x_i')}}{e^{S_{l}g(y_i')}}\cdot
\left(\beta^{-l}e^{-S_{l}f(x_i')}\right)^{s_i}\times\\&
\quad\quad\quad\quad
\sum\limits_{\substack{(\epsilon_{l+1},\ldots,\epsilon_{m_i})\in \Sigma_\beta^l\\ (\omega_{l+1},\ldots,\omega_{m_i})\in \Sigma_\beta^h}}\frac{e^{S_{h}f(\widetilde{x_i'})}}{e^{S_{h}g(\widetilde{y_i'})}}
\cdot\left(\beta^{-h}e^{-S_{h}f(\widetilde{x_i'})}\right)^{s_i}.
\end{align*}

The first part can be estimated as
\begin{align*}
\left(\beta^{-n_{i-1}-k_{i-1}}\right)^{s/(1+\epsilon)}\cdot\frac{e^{S_{l}f(x_i')}}{e^{S_{l}g(y_i')}}\cdot
\Big(\beta^{-l}e^{-S_{l}f(x_i')}\Big)^{s_i}&\leq\Big(\beta^{-(n_{i-1}+k_{i-1}+l)}\Big)^{s/(1+\epsilon)}\\ &=\Big(\beta^{-n}\Big)^{ s/(1+\epsilon)},
\end{align*}

since $$\frac{e^{S_{l}f(x_i')}}{e^{S_{l}g(y_i')}}\cdot\Big(e^{-S_{l}f(x_i')}\Big)^{s_i}\leq 1, \text{~for~} s_i\geq1.$$

To estimate the second part, we first recall that we defined  $s_i$ to be the solution of the equation
$$\sum\limits_{U,W\in \Sigma_{\beta_N}^{m_i}}
\frac{e^{S_nf(x_i')}}{e^{S_ng(y_i')}}\Big(\frac{1}{\beta^ne^{S_nf(x_i')}}\Big)^s=1.$$
Therefore,
\begin{align*}
1=\sum\limits_{U_1,W_1\in \Sigma_{\beta_N}^{l}}
\frac{e^{S_lf(x_i')}}{e^{S_lg(y_i')}}&\Big(\frac{1}{\beta^le^{S_lf(x_i')}}\Big)^{s_i}\times
\sum\limits_{U_2,W_2\in \Sigma_{\beta_N}^{h}}
\frac{e^{S_hf(\widetilde{x_i'})}}{e^{S_hg(\widetilde{y_i'})}}\Big(\frac{1}{\beta^le^{S_hf(\widetilde{x_i'})}}\Big)^{s_i}.
\end{align*}
So, with the similar arguments as in the paper \cite[ pp. 2095-2097]{TanWang} and \cite[pp. 1331-1332]{WW}, we derive that
$$\sum\limits_{U_2,W_2\in \Sigma_{\beta_N}^{h}}
\frac{e^{S_hf(\widetilde{x_i'})}}{e^{S_hg(\widetilde{y_i'})}}
\Big(\frac{1}{\beta^le^{S_hf(\widetilde{x_i'})}}\Big)^{s_i}\leq\beta^{l\epsilon}.$$

Therefore,$$\mu\Big(I_n(x)\times I_n(y)\Big)\leq\beta^{-n\cdot s/(1+\epsilon)}\cdot\beta^{l\epsilon}\leq(\beta^{-n})^{s/(1+\epsilon)-\epsilon}.$$

As far as the measure of a general ball $B(x,r)$ with $\beta^{-n-1}\leq r<\beta^{-n}$ is concerned, we notice that it can intersect at most $3$ cylinders of order $n$. Thus,$$\mu\Big(B(x,r)\Big)\leq3(\beta^{-n})^{s/(1+\epsilon)-\epsilon}\leq3\beta^sr^{s/(1+\epsilon)-\epsilon}\leq3\beta^2r^{s/(1+\epsilon)-\epsilon}.$$

So, finally, an application of the mass distribution principle (Proposition \ref{mdp}) yields that
$$\hdim \overline{E}(T_\beta, f,g)\geq s_0.$$

\subsubsection*{\bf Case II: $0\leq s_0\leq1$}

The arguments are similar to Case I but the calculations are different. In this case, for any $s<s_0\leq 1,$ it is trivial that
$$\frac{e^{S_nf(x')}}{e^{S_ng(y')}}\Big(\frac{1}{\beta^ne^{S_nf(x')}}\Big)^s\geq\big(\frac{1}{\beta^ne^{S_ng(y')}}\big)^s.$$
This means that the covering of the rectangle $J_n(U)\times J_n(W)$ by balls of larger side length is more preferable  and therefore, it reasonable to define the probability measure  of the rectangle to be the same
measure for the cylinder of order $n_i+l_i$.

Just like Case I, let $s_i$ be the solution to the equation
$$\sum\limits_{U,W\in \Sigma_{\beta_N}^{m_i}~~\text{ending with}~ 0^N}
\Big(\frac{1}{\beta^{m_i}e^{S_{m_i}g(y_i')}}\Big)^s=1,$$
where $y_i'\in I_{m_i}(W_i).$
By the continuity of the pressure function $P(T_\beta,f)$ with respect to $\beta$ we can assume that for all $m_i$  large enough we have that $s_i<1$ for all $i$ and
$|s_i-s_0|=o(1).$

We first define the measure $\mu$ on the basic cylinders.
\begin{itemize}
\item  Let $I_{n_1+k_1}(U_1,K_1)\times I_{n_1+l_1}(W_1,L_1)$ be a generic cylinder in $\mathcal{F}_1.$ Then define
$$\mu\Big(I_{n_1+k_1}(U_1,K_1)\times I_{n_1+l_1}(W_1,L_1)\Big)=\Big(\frac{1}{\beta^{m_1}e^{S_{m_1}g(y_1')}}\Big)^{s_1},$$
where $ y_1'\in I_{m_1}(W_1)$.
\end{itemize}

\begin{itemize}
\item Then the measure of it is evenly distributed on its sub-cylinders in $\mathcal{G}_1$. So, for a generic cylinder
$I_{n_1+k_1}(U_1,K_1)\times I_{n_1+k_1}(W_1,L_1,H_1)$ in $\mathcal{G}_1$, define

\begin{align*}
\mu\Big(I_{n_1+k_1}(U_1,K_1)\times I_{n_1+k_1}(W_1,L_1,H_1)\Big)&=\frac{1}{\#\Sigma_\beta^{k_1-l_1}}\Big(\frac{1}{\beta^{m_1}e^{S_{m_1}g(y_1')}}\Big)^{s_1}\\
&\asymp\frac{1}{\beta^{k_1-l_1}}\Big(\frac{1}{\beta^{m_1}e^{S_{m_1}g(y_1')}}\Big)^{s_1}.
\end{align*}

\end{itemize}

Assume that the measure on the cylinders of order $(i-1)$ has been well defined. Then to define the measure on the $i$th cylinder we proceed as follows.
\begin{itemize}
\item  Let $I_{n_i+k_i}(\Gamma_{i-1},U_i,K_i)\times I_{n_i+l_i}(\Upsilon_{i-1},W_i,L_i)$ be a generic
cylinder in $\mathcal{F}_i$.

Then define
\begin{align*}
\mu\Big(I_{n_i+k_i}(\Gamma_{i-1},U_i,K_i)\times I_{n_i+l_i}(\Upsilon_{i-1},&W_i,L_i)\Big)
= \mu\Big(I_{{n_{i-1}+k_{i-1}}}(\Gamma_{i-1})\\&\times I_{{n_{i-1}+k_{i-1}}}(\Upsilon_{i-1})\Big)\times\Big(\frac{1}{\beta^{m_i}e^{S_{m_i}g(y_i')}}\Big)^{s_i},
\end{align*}
where $ y_i'\in I_{m_i}(W_i)$.
\end{itemize}

\begin{itemize}
\item By the definition of $k_i,l_i$, the measure of a cylinder in $\mathcal{G}_i$ is then given as
\begin{align*}
 &\mu\Big(I_{n_i+k_i}(\Gamma_{i-1},U_i,K_i)\times I_{n_i+k_i}(\Upsilon_{i-1},W_i,L_i,H_i)\Big)\\
 &=\mu\Big(I_{{n_{i-1}+k_{i-1}}}(\Gamma_{i-1})\times I_{{n_{i-1}+k_{i-1}}}(\Upsilon_{i-1})\Big)\times\frac{1}{\#\Sigma_\beta^{k_i-l_i}}\times\Big(\frac{1}{\beta^{m_i}e^{S_{m_i}g(y_i')}}\Big)^{s_i}\\
 &\asymp \mu\Big(I_{{n_{i-1}+k_{i-1}}}(\Gamma_{i-1})\times I_{{n_{i-1}+k_{i-1}}}(\Upsilon_{i-1})\Big)\times\frac{e^{S_{m_i}g(y_i')}}{e^{S_{m_i}f(x_i')}}\times\Big(\frac{1}{\beta^{m_i}e^{S_{m_i}g(y_i')}}\Big)^{s_i}.
\end{align*}
\end{itemize}
\subsubsection{Estimation of the $\mu$-measure of cylinders.}
We first show by induction that for any $i\geq 1$ and a generic cylinder
$$I:=I_{{n_{i-1}+k_{i-1}}+m_{i}+k_i}(\Gamma_{i-1},U_i,K_i)\times I_{{n_{i-1}+k_{i-1}}+m_{i}+k_i}(\Upsilon_{i-1},W_i,L_i,H_i),$$
we have  $$\mu(I)\leq|I|^{s/(1+\epsilon)}.$$

When $i=1$. On the one hand, the length of $I$ is given as $$|I|=\beta^{-m_1-k_1}\geq\beta^{-m_1}\cdot\Big(e^{-S_{n_1}f(x_1')}\Big)^{1+\epsilon}=\beta^{-m_1}\cdot\Big(e^{-S_{m_1}f(x_1')}\Big)^{1+\epsilon}.$$
But on the other hand, by the definition of the measure $\mu$, it is clear that
\begin{align*}
\mu(I)&\leq\frac{e^{S_{m_1}g(y_1')}}{e^{S_{m_1}f(x_1')}}\cdot\Big(\frac{1}{\beta^{m_1}e^{S_{m_1}g(y_1')}}\Big)^{s_1}\\
&\leq\Big(\beta^{-m_1}e^{-S_{m_1}f(x_1')}\Big)^{s_1}\\
&\leq|I|^{s/(1+\epsilon)},
\end{align*}
by noting that $s_1<1.$

Just like Case I, we consider the inductive process. Assume that
\begin{align*}
\mu(I_{n_{i-1}+k_{i-1}}(\Gamma_{i-1})\times &I_{n_{i-1}+k_{i-1}}(\Upsilon_{i-1}))\leq|I_{n_{i-1}+k_{i-1}}(\Gamma_{i-1})\times I_{n_{i-1}+k_{i-1}}(\Upsilon_{i-1})|^{s/(1+\epsilon)}.
\end{align*}
 Let $$I=I_{n_i+k_i}(\Gamma_{i-1},U_i,K_i)\times I_{n_i+k_i}(\Upsilon_{i-1},W_i,L_i,H_i)$$ be a generic cylinder in $\mathcal{G}_i$.
 By (\ref{geq}) we get
$$|I|\geq|I_{n_{i-1}+k_{i-1}}(\Gamma_{i-1})\times I_{n_{i-1}+k_{i-1}}(\Upsilon_{i-1})|\times\Big(\beta^{-m_i}e^{-S_{m_i}f(x_i')}\Big)^{1+\epsilon}.$$

From the definition of the measure $\mu$, the induction and that $s_i<1$, it follows that
\begin{align*}
\mu(I)&=\mu\Big(I_{{n_{i-1}+k_{i-1}}}(\Gamma_{i-1})\times I_{{n_{i-1}+k_{i-1}}}(\Upsilon_{i-1})\Big)\times\frac{e^{S_{m_i}g(y_i')}}{e^{S_{m_i}f(x_i')}}
\times\Big(\frac{1}{\beta^{m_i}e^{S_{m_i}g(y_i')}}\Big)^{s_i}\\
&\leq|I_{n_{i-1}+k_{i-1}}(\Gamma_{i-1})\times I_{n_{i-1}+k_{i-1}}(\Upsilon_{i-1})|^{s/(1+\epsilon)}
\Big((\beta^{-m_i}e^{-S_{m_i}f(x_i')})^{1+\epsilon}\Big)^{s/(1+\epsilon)}\\
&\leq |I|^{s/(1+\epsilon)}\\ &=\Big(\beta^{-n_i-k_i}\Big)^{s/(1+\epsilon)}\\ &\asymp\Big(\beta^{-m_i-k_i}\Big)^{s/(1+\epsilon)}.
\end{align*}

So, for a rectangle $$J=I_{n_i+k_i}(\Gamma_{i-1},U_i,K_i)\times I_{n_i+l_i}(\Upsilon_{i-1},W_i,L_i)$$
in $\mathcal{F}_i,$ we have that
\begin{align*}
\mu(J)&=\mu\Big(I_{{n_{i-1}+k_{i-1}}}(\Gamma_{i-1})\times I_{{n_{i-1}+k_{i-1}}}(\Upsilon_{i-1})\Big)\times\Big(\frac{1}{\beta^{m_i}e^{S_{m_i}g(y_i')}}\Big)^{s_i}\\
&\leq \Big(\beta^{-n_{i-1}-k_{i-1}}\Big)^{s/(1+\epsilon)}\Big(\beta^{-m_i}\beta^{-l_i}\Big)^{s_i}\\
&\leq\Big(\beta^{-n_{i}-l_{i}}\Big)^{s/(1+\epsilon)}.
\end{align*}

For any $(x,y)\in \mathcal{F}_\infty,$ we will estimate the measure of $I_n(x)\times I_n(y)$ compared with its length $\beta^{-n}.$ By the construction of $\mathcal{F}_\infty,$ there exists $\{k_i,l_i\}_{i\geq1}$ such that for all $i\geq 1,$
$$(x,y)\in I_{n_i+k_i}(\Gamma_{i-1},U_i,K_i)\times I_{n_i+l_i}(\Upsilon_{i-1},W_i,L_i,).$$


For any $n\geq 1$, let $i\geq1$ be the integer such that
$$n_{i-1}+k_{i-1}<n\leq n_i+k_i=n_{i-1}+k_{i-1}+m_i+k_i.$$

\noindent{\bf Step I.} When $n_{i-1}+k_{i-1}+m_i+l_i\leq n \leq n_{i}+k_{i}=n_{i-1}+k_{i-1}+m_i+k_i.$

In this case, the cylinder can intersect only one rectangle in $\mathcal{F}_i,$ so
\begin{align*}
\mu\Big(I_n(x)\times I_n(y)\Big)
&=\mu\Big(I_{{n_{i}+k_{i}}}(\Gamma_{i-1},U_i,K_i)\times I_{{n_{i}+k_{i}}}(\Upsilon_{i-1},W_i,L_i)\Big)\\
&\leq\Big(\beta^{-n_{i}-l_{i}}\Big)^{s/(1+\epsilon)}
\\ &\leq \Big(\beta^{-n}\Big)^{s/(1+\epsilon)}.
\end{align*}

\noindent{\bf Step II.}  When $n_{i-1}+k_{i-1}+m_i\leq n \leq n_{i}+l_{i}=n_{i-1}+k_{i-1}+m_i+l_i.$

Then the cylinder $I_n(x)\times I_n(y)$ contains $\beta^{n_{i}+l_{i}-n}$ cylinders in $\mathcal{F}_i$ with order $n_{i}+l_{i}$. Note that by the definition of
$\{k_j,l_j\}_{1\leq j\leq i}$, the first $i$-pairs $\{k_j,l_j\}_{1\leq j\leq i}$ depends only on the first $n_i$ digits of $(x,y)$. So the measure of the sub-cylinder of order $n_{i}+k_{i}$
are the same. So, its measure of $I_n(x)\times I_n(y)$ can be estimated as
\begin{align*}
\mu\Big(I_n(x)\times I_n(y)\Big)
&=\mu\Big(I_{{n_{i}+k_{i}}}(\Gamma_{i-1},U_i,K_i)\times I_{{n_{i}+k_{i}}}(\Upsilon_{i-1},W_i,L_i)\Big)\times \frac{1}{\beta^{n-n_{i}-l_{i}}}\\
&\leq \Big(\beta^{-n_{i}-l_{i}}\Big)^{s/(1+\epsilon)}\times \frac{1}{\beta^{n-n_{i}-l_{i}}}
\\ &\leq\Big(\beta^{-n}\Big)^{s/(1+\epsilon)}.
\end{align*}

\noindent{\bf Step III.}  When $n_{i-1}+k_{i-1}\leq n \leq n_{i-1}+k_{i-1}+m_i.$

Assume $U_i=(\epsilon_1,\epsilon_2,\ldots,\epsilon_{m_i}), W_i=(\omega_1,\omega_2,\ldots,\omega_{m_i}).$ Write $l=n-(n_{i-1}+k_{i-1})$ and
$h=m_i-l$. Then
\begin{align*}
&\mu\Big(I_n(x)\times I_n(y)\Big)\\ &=\sum\limits_{\substack{ (\epsilon_{l+1},\ldots,\epsilon_{m_i})\in \Sigma_\beta^l\\ (\omega_{l+1},\ldots,\omega_{m_i})\in \Sigma_\beta^h}}
\mu\Big(I_{n_i+k_i}(\Gamma_{i-1},U_i,K_i)
\times I_{n_i+l_i}(\Upsilon_{i-1},W_i,L_i)\Big)\\
&=\mu\Big(I_{{n_{i-1}+k_{i-1}}}(\Gamma_{i-1})\times I_{{n_{i-1}+k_{i-1}}}(\Upsilon_{i-1})\Big)\times
\sum\limits_{\substack{(\epsilon_{l+1},\ldots,\epsilon_{m_i})\in \Sigma_\beta^l\\ (\omega_{l+1},\ldots,\omega_{m_i})\in \Sigma_\beta^h}}
\Big(\beta^{-m_i}e^{-S_{m_i}g(y_i')}\Big)^{s_i}.
\end{align*}
Then by the estimation on the measure of cylinders of order $n_{i-1}+k_{i-1}$ and let $\widetilde{y_i'}=T_\beta^ly_i'$, we get
\begin{align*}
\mu\Big(I_n(x)\times I_n(y)\Big)&\leq\Big(\beta^{-n_{i-1}-k_{i-1}}\Big)^{s/(1+\epsilon)}\cdot\Big(\beta^{-l}e^{-S_{l}g(y_i')}\Big)^{s_i}\times
\sum\limits_{\substack{(\epsilon_{l+1},\ldots,\epsilon_{m_i})\in \Sigma_\beta^l\\ (\omega_{l+1},\ldots,\omega_{m_i})\in \Sigma_\beta^h}}\Big(\beta^{-h}e^{-S_{h}g(\widetilde{y_i'})}\Big)^{s_i}\\
&\leq\Big(\beta^{-n}\Big)^{s/(1+\epsilon)}\cdot\sum\limits_{\substack{(\epsilon_{l+1},\ldots,\epsilon_{m_i})\in \Sigma_\beta^l\\ (\omega_{l+1},\ldots,\omega_{m_i})\in \Sigma_\beta^h}}
\Big(\beta^{-h}e^{-S_{h}g(\widetilde{y_i'})}\Big)^{s_i}.
\end{align*}

Recall the definition of $s_i:$
$$\sum\limits_{U,W\in \Sigma_{\beta_N}^{m_i}}
\Big(\frac{1}{\beta^{m_i}e^{S_{m_i}g(y_i')}}\Big)^s=1.$$
Then
\begin{align*}
1=\sum\limits_{U_1,W_1\in \Sigma_{\beta_N}^{l}}
\Big(\frac{1}{\beta^le^{S_lg(y_1')}}\Big)^{s_i}\cdot
\sum\limits_{U_2,W_2\in \Sigma_{\beta_N}^{h}}
\Big(\frac{1}{\beta^le^{S_hg(\widetilde{y_1'})}}\Big)^{s_i},
\end{align*}
where $y_1'^*\in I_{l}(U_1),\widetilde{y_1'}\in I_{h}(W_2).$\\
So, with the similar argument as in the previous section, we have that
$$\sum\limits_{U_2,W_2\in \Sigma_{\beta_N}^{l} }
\Big(\frac{1}{\beta^he^{S_hg(\widetilde{y_h'})}}\Big)^{s_i}\leq\beta^{l\epsilon}.$$

Therefore,$$\mu\Big(I_n(x)\times I_n(y)\Big)\leq\Big(\beta^{-n}\Big)^ {s/(1+\epsilon)}\cdot\beta^{l\epsilon}\leq\Big(\beta^{-n}\Big)^{s/(1+\epsilon)-\epsilon}.$$

Notice that a general ball $B(x,r)$ with $\beta^{-n-1}\leq r<\beta^{-n}$ can intersect at most $3$ cylinders of order $n$. Therefore the measure of the general ball can be estimated as,
$$\mu\Big(B(x,r)\Big)\leq3\Big(\beta^{-n}\Big)^{s/(1+\epsilon)-\epsilon}\leq3\beta^sr^{s/(1+\epsilon)-\epsilon}\leq3\beta^2r^{s/(1+\epsilon)-\epsilon}.$$

So, finally, by using the mass distribution principle we have the lower bound of the Hausdorff dimension of this case,
$$\hdim \overline{E}(T_\beta, f,g)\geq s_0.$$
Hence combining both {cases}, we have the desired conclusion.\\

\noindent{\bf Acknowledgments.} We would like to thank Professor Baowei Wang for useful discussions on this project. The first-named author was supported by  the ARCDP200100994. {  The Second author was supported by Natural Science Research Project of West Anhui University (No. WGKQ2021020) and Provincial Natural Science Research Project of Anhui Colleges (No. KJ2021A0950).}


\end{document}